\documentclass[12pt, reqno]{amsart}
\usepackage{ amsmath,amsthm, amscd, amsfonts, amssymb, graphicx, color}
\usepackage[bookmarksnumbered, colorlinks, plainpages]{hyperref}
\newtheorem{thm}{Theorem}[section]
\newtheorem{cor}[thm]{Corollary}
\newtheorem{lem}[thm]{Lemma}

\newtheorem{exam}[thm]{Example}
\numberwithin{equation}{section}

\begin{document}

\title{The g-Drazin inverses of anti-triangular block operator matrices}

\author{Huanyin Chen}
\author{Marjan Sheibani}
\address{
Department of Mathematics\\ Hangzhou Normal University\\ Hang -zhou, China}
\email{<huanyinchen@aliyun.com>}
\address{Women's University of Semnan (Farzanegan), Semnan, Iran}
\email{<sheibani@fgusem.ac.ir>}


\begin{abstract}
An element $a$ in a Banach algebra $\mathcal{A}$ has g-Drazin inverse if there exists $b\in \mathcal{A}$ such that $ab=ba, b=bab$ and $a-a^2b
 \in \mathcal{A}^{qnil}$. In this paper we find new explicit representations of the g-Drazin inverse of the block operator matrix $\left(
  \begin{array}{cc}
    E&I\\
    F&0
  \end{array}
\right)$. We thereby solve a wider kind of singular differential equations posed by Campbell [S.L. Campbell, The Drazin inverse and systems of second order linear differential equations, Linear $\&$ Multilinear Algebra, 14(1983), 195--198].\end{abstract}

\maketitle

\section{Introduction}

Let $\mathcal{A}$ be a complex Banach algebra with an identity. An element $a$ in a Banach algebra $\mathcal{A}$ has g-Drazin inverse provided that there exists some $b\in \mathcal{A}$ such that $b=bab, ab=ba, a-a^2b \in \mathcal{A}^{qnil}$. Here, $\mathcal{A}^{qnil}=\{ a\in \mathcal{A}~\mid~ \lim\limits_{n\to\infty}\parallel a^n\parallel^{\frac{1}{n}}=0\}.$ Such $b$ is unique, if it exists, and we denote it by $a^d$.
We always use $\mathcal{A}^{d}$ to stand for the set of all g-Drazin invertible elements in $\mathcal{A}$.
We say $a$ has Drazin inverse if there exists some $b\in \mathcal{A}$ such that $b=bab, ab=ba, a^n=a^{n+1}b$. Such $b$ is unique,
if exists, and we denote it by $a^D$. The smallest integer $n$ satisfying the preceding equations is called the index of $a$.
Evidently, every Drazin invertible $a\in \mathcal{A}$ has g-Drazin inverse and $a^D=a^d$.

Let $E,F$ be block operator matrices and $I$ be the identity matrix over a Banach space $X$. It is attractive to investigate the Drazin inverse of the block matrix $M=\left(
  \begin{array}{cc}
    E&I\\
    F&0
  \end{array}
\right)$. The relationship of computing the Drazin inverse of the block complex matrices $\left(
  \begin{array}{cc}
    E&-F\\
    -I&0
  \end{array}
\right)$ to second order differential equations was observed by Campbell (see~\cite{C}). The application of Drazin inverse to singular
difference equations was also found in ~\cite{B}.
By using Cline's formula (see~\cite[Theorem 2.1]{LCC}), we easily see that $$\left(
  \begin{array}{cc}
    E&-F\\
    -I&0
  \end{array}
\right)^D=\left(
  \begin{array}{cc}
    E&I\\
    -I&0
  \end{array}
\right)\big(\left(
  \begin{array}{cc}
    E&I\\
    F&0
  \end{array}
\right)^D\big)^2\left(
  \begin{array}{cc}
  I&0\\
  0&-F
  \end{array}
\right).$$ Thus finding the solution of singular differential equation $Ax''(t)+Bx(t)+Cx(t)=0$ is equivalent to finding the explicit representation of the Drazin inverse for the block complex matrix $M$.

In 2005, Castro-Gonz\'{a}lez and Dopazo considered the representations of the Drazin inverse for a class of block matrix $\left(
  \begin{array}{cc}
    I&I\\
    F&0
  \end{array}
\right)$ (see~\cite{CD}). In 2011, Bu et al. gave the Drazin inverse of the preceding block matrix $M$ under $EF=FE$ (see~\cite{B}).
Patricio and Hartwig investigated the g-Drazin inverse of $M$ under $F^{\pi}EFF^d=0, F^{\pi}FE=EFF^{\pi}$ (see ~\cite{P}).
Here, $F^{\pi}=I-FF^d$. In 2016, Zhang studied the g-Drazin inverse of $M$ under $F^dEF^{\pi}=0, F^{\pi}FE=0$ and $F^{\pi}EF^d=0, EFF^{\pi}=0$ (see~\cite[Theorem 2.6, Theorem 2.8]{Z}).

In~\cite[Theorem 3.1]{DY}, Deng and Wei gave formula for the g-Drazin inverse of $M$ under $EF=0$. In~\cite[Theorem 2.2]{CI}, Cvetkovi\'{c}-Ili\'{c}
obtained the representation of Drazin inverse $M^D$ under certain conditions $P_n=0$. The motivation of this paper is to present formulae for the g-Drazin inverse of $M$ under new conditions. We shall give formula for the g-Drazin inverse $M^d$ under $EF^2=0, EFE=0$. Then we give the explicit representation of $M^d$ under the condition $F^{\pi}EF^2=0$ and $F^{\pi}EFE=0$. In particular, we obtain new representations of $M^D$ for an anti-triangular complex matrix $M$. We thereby solve a wider kind of singular differential equations posed by Campbell (see~\cite{C}).

Throughout the paper, all matrices are bounded linear operator matrices over a Banach space $X$.
Let $a\in \mathcal{A}$ and $e^2=e\in \mathcal{A}$. Then $a$ has the Pierce decomposition relative to $e$, and we denote it by $\left(\begin{array}{cc}
a_{11}&a_{12}\\
a_{21}&a_{22}
\end{array}
\right)_e$. $[x]$ stands for the truncates integer of $x$.

\section{main results}

We begin with the following elementary lemma.

\begin{lem} (~\cite[Lemma 2.2]{Z}) Let $A, B, C$ are operators over a Banach space $X$. If $A$ and $B$ have g-Drazin inverses, then $M=\left(\begin{array}{cc}
A&0\\
C&B
\end{array}
\right)$ has g-Drazin inverse and
$$M^d=\left(\begin{array}{cc}
A^d&0\\
X&B^d
\end{array}
\right),$$ where $$X=(B^d)^2\big(\sum\limits_{i=0}^{\infty}(B^d)^iCA^i\big)A^{\pi}+B^{\pi}\big(\sum\limits_{i=0}^{\infty}B^iC(A^d)^i\big)(A^d)^2-B^dCA^d.$$
\end{lem}

\begin{lem} Let $P$ and $Q$ have g-Drazin inverses. If $PQ^2=0, PQP=0$, then $P+Q$ has g-Drazin inverse and $$\begin{array}{ll}
&(P+Q)^d\\
=&Q^{\pi}\sum\limits_{i=0}^{\infty}Q^i(P^d)^{i+1}+\sum\limits_{i=0}^{\infty}(Q^d)^{i+1}P^iP^{\pi}+Q^{\pi}\sum\limits_{i=0}^{\infty}Q^i(P^d)^{i+2}Q\\
+&\sum\limits_{i=0}^{\infty}(Q^d)^{i+3}P^{i+1}P^{\pi}Q-Q^dP^dQ-(Q^d)^2PP^dQ.
\end{array}$$
\end{lem}\begin{proof} This was proved in ~\cite[Theorem 2.1]{Y} for Drazin inverse of matrices, the proof is similar for g-Drazin inverse of operators.\end{proof}

We are ready to prove the following.

\begin{thm} Let $E$ and $F$ have g-Drazin inverses. If $EF^2=0$ and $EFE=0$, then $M=\left(
  \begin{array}{cc}
    E&I\\
    F&0
  \end{array}
\right)$ has g-Drazin inverse. In this case, $$M^d=\left(\begin{array}{cc}
\Gamma&\Delta\\
 \Lambda&\Xi
\end{array}
\right),$$ where $${\scriptsize \begin{array}{l}
\Gamma=E^d+(E^d)^3F-F^dE^dF-(F^d)^2E^2E^dF-FF^dE^d+F^dEE^{\pi}\\
+(F^d)^2EF-FF^d(E^d)^3F+(F^d)^3E^3E^{\pi}F+\sum\limits_{i=1}^{\infty}F^{i}F^{\pi}E^{2i+1}\\
+\sum\limits_{i=1}^{\infty}(F^d)^{i+1}E^{2i+1}E^{\pi}+\sum\limits_{i=1}^{\infty}F^{i}F^{\pi}E^{2i}E^dF+\sum\limits_{i=1}^{\infty}(F^d)^{i+3}E^{2i+3}E^{\pi}F,\\
\Delta=(E^d)^2+(E^d)^4F-F^d(E^d)^2F-(F^d)^2EE^dF-FF^d(E^d)^2\\
+F^dE^{\pi}-FF^d(E^d)^4F+(F^d)^3E^2E^{\pi}F+\sum\limits_{i=1}^{\infty}F^iF^{\pi}E^{2i+1}E^d\\
+\sum\limits_{i=1}^{\infty}(F^d)^{i+1}E^{2i}E^{\pi}+\sum\limits_{i=1}^{\infty}F^iF^{\pi}E^{2i-1}E^dF+
\sum\limits_{i=1}^{\infty}(F^d)^{i+3}E^{2i+2}E^{\pi}F;\\
\Lambda=FF^{\pi}(E^d)^2+FF^dE^{\pi}+FF^{\pi}(E^d)^4F+(F^d)^2E^2E^{\pi}F\\
-FF^d(E^d)^2F-F^dEE^dF+\sum\limits_{i=1}^{\infty}F^{i+1}F^{\pi}(E^d)^{2i+2}+\sum\limits_{i=1}^{\infty}(F^d)^iE^{2i}E^{\pi}\\
+\sum\limits_{i=1}^{\infty}F^{i+1}F^{\pi}(E^d)^{2i+4}F+\sum\limits_{i=1}^{\infty}(F^d)^{i+2}E^{2i+2}E^{\pi}F,\\
\Xi=-FF^d(E^d)^3F-F^dE^dF+FF^{\pi}(E^d)^3-FF^dE^d\\
+FF^{\pi}(E^d)^5F-(F^d)^2E^2E^dF+(F^d)^2EF+\sum\limits_{i=1}^{\infty}F^{i+1}F^{\pi}(E^d)^{2i+3}\\
+\sum\limits_{i=1}^{\infty}(F^d)^iE^{2i-1}E^{\pi}+\sum\limits_{i=1}^{\infty}F^{i+1}F^{\pi}(E^d)^{2i+5}F+\sum\limits_{i=1}^{\infty}(F^d)^{i+2}E^{2i+1}E^{\pi}F~~~~~(*)
\end{array}}$$\end{thm}
\begin{proof} Clearly, we have $$M^2=\left(\begin{array}{cc}
E^2+F&E\\
FE&F
\end{array}
\right)=P+Q,$$ where $$P=\left(\begin{array}{cc}
E^2&E\\
0&0
\end{array}
\right), Q=\left(\begin{array}{cc}
F&0\\
FE&F
\end{array}
\right).$$ In light of Lemma 2.1, we have
$$\begin{array}{c}
P^d=\left(\begin{array}{cc}
(E^d)^{2}&(E^d)^{3}\\
0&0
\end{array}
\right),P^{\pi}=\left(\begin{array}{cc}
E^{\pi}&-E^d\\
0&I
\end{array}
\right);\\
Q^d=\left(\begin{array}{cc}
F^d&0\\
X&F^d
\end{array}
\right); Q^{\pi}=\left(\begin{array}{cc}
F^{\pi}&0\\
-FX&F^{\pi}
\end{array}
\right),
\end{array}$$
where $$\begin{array}{lll}
X&=&F^{\pi}\sum\limits_{i=0}^{\infty}F^{i+1}E(F^d)^{i+2}+\sum\limits_{i=0}^{\infty}(F^d)^{i+2}FEF^{i}F^{\pi}\\
&=&(F^d)^2FE+(F^d)^3FEF\\
&=&F^dE+(F^d)^2EF.
\end{array}$$
For any $n\geq 2$, we have $$\begin{array}{c}
P^n=\left(\begin{array}{cc}
E^{2n}&E^{2n-1}\\
0&0
\end{array}
\right); (P^d)^n=\left(\begin{array}{cc}
(E^d)^{2n}&(E^d)^{2n+1}\\
0&0
\end{array}
\right);\\
Q^n=\left(\begin{array}{cc}
F^n&0\\
Y_n&F^n,
\end{array}
\right), Y_1=FE, Y_n=Y_{n-1}F+F^nE;\\
(Q^d)^n=\left(\begin{array}{cc}
(F^d)^n&0\\
X_n&(F^d)^n
\end{array}
\right), X_1=X, X_n=X_{n-1}F^d+F^dX_{n-1}.
\end{array}$$
Since $EF^2=0$ and $EFE=0$, by induction, we verify that $$X_n=(F^d)^nE+(F^d)^{n+1}EF, Y_n=F^{n-1}EF+F^nE.$$
Moreover, we check that $$PQ^2=0, PQP=0.$$ In view of Lemma 2.2, $M^2=P+Q$ has g-Drazin inverse.
Further, we get $$\begin{array}{ll}
&(M^2)^d=(P+Q)^d\\
=&Q^{\pi}\sum\limits_{i=0}^{\infty}Q^i(P^d)^{i+1}+\sum\limits_{i=0}^{\infty}(Q^d)^{i+1}P^iP^{\pi}+Q^{\pi}\sum\limits_{i=0}^{\infty}Q^i(P^d)^{i+2}Q\\
+&\sum\limits_{i=0}^{\infty}(Q^d)^{i+3}P^{i+1}P^{\pi}Q-Q^dP^dQ-(Q^d)^2PP^dQ.
\end{array}$$
We easily check that
$$\begin{array}{c}
Q^dP^dQ=\left(\begin{array}{cc}
F^d(E^d)^2F&F^d(E^d)^3F\\
X(E^d)^2F&X(E^d)^3F
\end{array}
\right);\\
(Q^d)^2PP^dQ=\left(\begin{array}{cc}
(F^d)^2EE^dF&(F^d)^2E^dF\\
X_2EE^dF&X_2E^dF
\end{array}
\right).
\end{array}$$
If $i\geq 1$, we have
$${\scriptsize \begin{array}{l}
Q^{\pi}Q^i(P^d)^{i+1}=\\
\left(\begin{array}{cc}
F^{\pi}F^i(E^d)^{2i+2}&F^{\pi}F^i(E^d)^{2i+3}\\
-FXF^i(E^d)^{2i+2}+F^{\pi}Y_i(E^d)^{2i+2}&-FXF^i(E^d)^{2i+3}+F^{\pi}Y_i(E^d)^{2i+3}
\end{array}
\right);
\end{array}}$$
$${\scriptsize Q^{\pi}P^d=\left(\begin{array}{cc}
F^{\pi}(E^d)^{2}&F^{\pi}(E^d)^{3}\\
-FX(E^d)^{2}&-FX(E^d)^{3}
\end{array}
\right).}$$
If $i\geq 1$, we have
$${\scriptsize (Q^d)^{i+1}P^iP^{\pi}=\left(\begin{array}{cc}
(F^d)^{i+1}E^{2i}E^{\pi}&-(F^d)^{i+1}E^{2i}E^d+(F^d)^{i+1}E^{2i-1}\\
X_{i+1}E^{2i}E^{\pi}&-X_{i+1}E^{2i}E^d+X_{i+1}E^{2i-1}
\end{array}
\right);}$$
$${\scriptsize Q^dP^{\pi}=\left(\begin{array}{cc}
F^dE^{\pi}&-F^dE^d\\
XE^{\pi}&-XE^d+F^d
\end{array}
\right);}$$
If $i\geq 1$, we have
$${\scriptsize \begin{array}{l}
Q^{\pi}Q^i(P^d)^{i+2}Q=\\
\left(\begin{array}{cc}
F^{\pi}F^i(E^d)^{2i+4}F&F^{\pi}F^i(E^d)^{2i+5}F\\
-FXF^i(E^d)^{2i+4}F+F^{\pi}Y_i(E^d)^{2i+4}F&-FXF^i(E^d)^{2i+5}F+F^{\pi}Y_i(E^d)^{2i+5}F
\end{array}
\right);
\end{array}}$$
$${\scriptsize Q^{\pi}(P^d)^{2}Q=\left(\begin{array}{cc}
F^{\pi}(E^d)^{4}F&F^{\pi}(E^d)^{5}F\\
-FX(E^d)^{4}F&-FX(E^d)^{5}F
\end{array}
\right);}$$
$${\scriptsize \begin{array}{l}
(Q^d)^{i+3}P^{i+1}P^{\pi}Q=\\
\left(\begin{array}{cc}
(F^d)^{i+3}E^{2i+2}E^{\pi}F&-(F^d)^{i+3}E^{2i+2}E^dF+(F^d)^{i+3}E^{2i+1}F\\
X_{i+3}E^{2i+2}E^{\pi}F&-X_{i+3}E^{2i+2}E^dF+X_{i+3}E^{2i+1}F
\end{array}
\right).
\end{array}}$$
Since $XE^d=F^dEE^d$ and $X_2=(F^d)^2E+(F^d)^3EF$, we have
$$\begin{array}{c}
Q^dP^dQ=\left(\begin{array}{cc}
F^d(E^d)^2F&F^d(E^d)^3F\\
F^dE^dF&F^d(E^d)^2F
\end{array}
\right);\\
(Q^d)^2PP^dQ=\left(\begin{array}{cc}
(F^d)^2EE^dF&(F^d)^2E^dF\\
(F^d)^2E^2E^dF&(F^d)^2EE^dF
\end{array}
\right).
\end{array}$$
Moreover, we have
$$\begin{array}{c}
Q^{\pi}P^d=\left(\begin{array}{cc}
F^{\pi}(E^d)^{2}&F^{\pi}(E^d)^{3}\\
-FF^dE^d&-FF^d(E^d)^{2}
\end{array}
\right);\\
Q^dP^{\pi}=\left(\begin{array}{cc}
F^dE^{\pi}&-F^dE^d\\
F^dEE^{\pi}+(F^d)^2EF&F^dE^{\pi}
\end{array}
\right).
\end{array}$$
Also we check that
$$\begin{array}{c}
FXF^i(E^d)^{2i+2}=\{\begin{array}{ll}
FF^dE^d &i=0;\\
0&i\geq 1;
\end{array}\\
F^{\pi}Y_i(E^d)^{2i+2}=F^{\pi}F^iE^{2i+1} ~(i\geq 1);\\
X_{i+1}E^{2i-1}=(F^d)^{i+1}E^{2i} ~(i\geq 1);\\
X_{i+3}E^{2i+1}=(F^d)^{i+3}E^{2i+2} ~(i\geq 0).
\end{array}$$
By virtue of~\cite[Corollary 2.2]{M}, $M$ has g-Drazin inverse. Therefore
$$M^d=M(M^d)^2=M(M^2)^d=\left(
  \begin{array}{cc}
    E&I\\
    F&0
  \end{array}
\right)(M^2)^d=\left(\begin{array}{cc}
\Gamma&\Delta\\
 \Lambda&\Xi
\end{array}
\right),$$ where $\Gamma, \Delta, \Lambda$ and $\Xi$ are given as in $(*)$ by direct computation.\end{proof}

For a block complex matrix, the Drazin and g-Drazin inverses coincide with each other, and so we derive

\begin{cor} Let $E$ and $F$ be $n\times n$ complex matrices, and let $M=\left(
  \begin{array}{cc}
    E&I\\
    F&0
  \end{array}
\right)$. If $EF^2=0$ and $EFE=0$, then $$M^D=\left(\begin{array}{cc}
\Gamma&\Delta\\
 \Lambda&\Xi
\end{array}
\right),$$ where $${\scriptsize \begin{array}{l}
\Gamma=E^D+(E^D)^3F-F^DE^DF-(F^D)^2E^2E^DF-FF^DE^D+F^DEE^{\pi}\\
+(F^D)^2EF-FF^D(E^D)^3F+(F^D)^3E^3E^{\pi}F+\sum\limits_{i=1}^{t}F^{i}F^{\pi}E^{2i+1}\\
+\sum\limits_{i=1}^{[\frac{s}{2}]}(F^D)^{i+1}E^{2i+1}E^{\pi}+\sum\limits_{i=1}^{t}F^{i}F^{\pi}E^{2i}E^DF+\sum\limits_{i=1}^{[\frac{s}{2}]-1}(F^D)^{i+3}E^{2i+3}E^{\pi}F,\\
\Delta=(E^D)^2+(E^D)^4F-F^D(E^D)^2F-(F^D)^2EE^DF-FF^D(E^D)^2\\
+F^DE^{\pi}-FF^D(E^D)^4F+(F^D)^3E^2E^{\pi}F+\sum\limits_{i=1}^{t}F^iF^{\pi}E^{2i+1}E^D\\
+\sum\limits_{i=1}^{[\frac{s}{2}]+1}(F^D)^{i+1}E^{2i}E^{\pi}+\sum\limits_{i=1}^{t}F^iF^{\pi}E^{2i-1}E^DF+
\sum\limits_{i=1}^{\frac{s}{2}}(F^D)^{i+3}E^{2i+2}E^{\pi}F;\\
\Lambda=FF^{\pi}(E^D)^2+FF^DE^{\pi}+FF^{\pi}(E^D)^4F+(F^D)^2E^2E^{\pi}F\\
-FF^D(E^D)^2F-F^DEE^DF+\sum\limits_{i=1}^{t-1}F^{i+1}F^{\pi}(E^D)^{2i+2}+\sum\limits_{i=1}^{[\frac{s}{2}]+1}(F^D)^iE^{2i}E^{\pi}\\
+\sum\limits_{i=1}^{t-1}F^{i+1}F^{\pi}(E^D)^{2i+4}F+\sum\limits_{i=1}^{[\frac{s}{2}]}(F^D)^{i+2}E^{2i+2}E^{\pi}F,\\
\Xi=-FF^D(E^D)^3F-F^DE^DF+FF^{\pi}(E^D)^3-FF^DE^D\\
+FF^{\pi}(E^D)^5F-(F^D)^2E^2E^DF+(F^D)^2EF+\sum\limits_{i=1}^{t-1}F^{i+1}F^{\pi}(E^D)^{2i+3}\\
+\sum\limits_{i=1}^{[\frac{s}{2}]+1}(F^D)^iE^{2i-1}E^{\pi}+\sum\limits_{i=1}^{t-1}F^{i+1}F^{\pi}(E^D)^{2i+5}F+\sum\limits_{i=1}^{[\frac{s}{2}]}(F^D)^{i+2}E^{2i+1}E^{\pi}F;
\end{array}}$$ $s=i(E)$ and $t=i(F)$.\end{cor}
\begin{proof} Since $E^sE^{\pi}=F^tF^{\pi}=0$, we obtain the result by Theorem 2.3.\end{proof}

We give a numerical example to illustrate the representation of Corollary 2.4.

\begin{exam} Let $M=\left(
  \begin{array}{cc}
    E&I\\
    F&0
  \end{array}
\right)$, where $$E=\left(\begin{array}{cc}
1&0\\
0&0
\end{array}
\right), F=\left(\begin{array}{cc}
0&1\\
0&0
\end{array}
\right)\in M_2({\Bbb C}).$$ Then $EF^2=0$ and $EFE=0$, while $EF\neq 0$. In this case, $$M^D=\left(\begin{array}{cccc}
1&1&1&1\\
0&0&0&0\\
0&0&0&0\\
0&0&0&0
\end{array}
\right).$$\end{exam}
\begin{proof} Clearly, we have $$EF^2=0, EFE=0~\mbox{and}~EF=F\neq 0.$$ To compute the Drazin inverse of $M$,
we shall apply Corollary 2.4 to $M$. One easily verifies that $$F^2=F^D=FE=0, E^2=E.$$ Using the notations in Corollary 2.4, we directly compute
$$\Gamma=\Delta=E+EF=\left(\begin{array}{cc}
1&1\\
0&0
\end{array}
\right)~\mbox{and}~\Lambda=\Xi=0,$$ as required.
\end{proof}

We have accumulated all the information necessary to prove the main result.

\begin{thm} Let $E,F,EF^{\pi}$ have g-Drazin inverses. If $F^{\pi}EF^2$ $=0$ and $F^{\pi}EFE=0$, then $M=\left(
  \begin{array}{cc}
    E&I\\
    F&0
  \end{array}
\right)$ has g-Drazin inverse. In this case, $$\begin{array}{lll}
M^d&=&\left(
  \begin{array}{cc}
   \Theta&\Psi\\
    \Phi&\Omega
  \end{array}
\right)+\sum\limits_{i=0}^{\infty}\left(
  \begin{array}{cc}
    0&F^d\\
   FF^d&-EF^d
  \end{array}
\right)^{i+1}\\
&&\left(
  \begin{array}{cc}
   F^{\pi}E+FF^dEF^{\pi}&F^{\pi}\\
    FF^{\pi}&0
  \end{array}
\right)^i\left(
  \begin{array}{cc}
   I-E\Theta&-E\Psi\\
    -F\Theta&I-F\Psi
  \end{array}
\right),
\end{array}$$ where
$${\scriptsize\begin{array}{lll}
\Theta&=&\big[F^{\pi}E^d+F^{\pi}(E^d)^3FF^{\pi}+FF^{\pi}(E^d)^3+FF^{\pi}(E^d)^5FF^{\pi}\\
&+&\sum\limits_{i=1}^{\infty}F^{i+1}F^{\pi}(E^d)^{2i+3}+\sum\limits_{i=1}^{\infty}F^{i+1}F^{\pi}(E^d)^{2i+5}FF^{\pi}\big]\\
&&\big[F^{\pi}EE^d+F^{\pi}(E^d)^2FF^{\pi}+FF^{\pi}(E^d)^2+FF^{\pi}(E^d)^4FF^{\pi}\\
&+&\sum\limits_{i=1}^{\infty}F^{i+1}F^{\pi}(E^d)^{2i+2}+\sum\limits_{i=1}^{\infty}F^{i+1}F^{\pi}(E^d)^{2i+4}FF^{\pi}\big]\\
&+&\big[F^{\pi}EE^d+F^{\pi}(E^d)^2FF^{\pi}+FF^{\pi}(E^d)^2+FF^{\pi}(E^d)^4FF^{\pi}\\
&-&F^{\pi}E^dF^{\pi}E-F^{\pi}(E^d)^3FF^{\pi}E-FF^{\pi}(E^d)^3F^{\pi}E-FF^{\pi}(E^d)^5FF^{\pi}E\\
&+&\sum\limits_{i=1}^{\infty}F^{i+1}F^{\pi}(E^d)^{2i+2}+\sum\limits_{i=1}^{\infty}F^{i+1}F^{\pi}(E^d)^{2i+4}FF^{\pi}\\
&-&\sum\limits_{i=1}^{\infty}F^{i+1}F^{\pi}(E^d)^{2i+3}F^{\pi}E-\sum\limits_{i=1}^{\infty}F^{i+1}F^{\pi}(E^d)^{2i+5}FF^{\pi}E\big]\\
&&\big[F^{\pi}E^d+F^{\pi}(E^d)^3FF^{\pi}+\sum\limits_{i=1}^{\infty}F^iF^{\pi}E^{2i+1}+\sum\limits_{i=1}^{\infty}F^iF^{\pi}E^{2i}E^dFF^{\pi}\big],\\
\end{array}}$$
$${\scriptsize\begin{array}{lll}
\Psi&=&\big[F^{\pi}E^dF^{\pi}+F^{\pi}(E^d)^3FF^{\pi}+FF^{\pi}(E^d)^3F^{\pi}+FF^{\pi}(E^d)^5FF^{\pi}\\
&+&\sum\limits_{i=1}^{\infty}F^{i+1}F^{\pi}(E^d)^{2i+3}F^{\pi}+\sum\limits_{i=1}^{\infty}F^{i+1}F^{\pi}(E^d)^{2i+5}FF^{\pi}\big]^2\\
&+&\big[F^{\pi}EE^d+F^{\pi}(E^d)^2FF^{\pi}+FF^{\pi}(E^d)^2+FF^{\pi}(E^d)^4FF^{\pi}\\
&-&F^{\pi}E^dF^{\pi}E-F^{\pi}(E^d)^3FF^{\pi}E-FF^{\pi}(E^d)^3F^{\pi}E-FF^{\pi}(E^d)^5FF^{\pi}E\\
&+&\sum\limits_{i=1}^{\infty}F^{i+1}F^{\pi}(E^d)^{2i+2}+\sum\limits_{i=1}^{\infty}F^{i+1}F^{\pi}(E^d)^{2i+4}FF^{\pi}\\
&-&\sum\limits_{i=1}^{\infty}F^{i+1}F^{\pi}(E^d)^{2i+3}F^{\pi}E-\sum\limits_{i=1}^{\infty}F^{i+1}F^{\pi}(E^d)^{2i+5}FF^{\pi}E\big]\\
&&\big[F^{\pi}(E^d)^2F^{\pi}+F^{\pi}(E^d)^4FF^{\pi}+\sum\limits_{i=1}^{\infty}F^iF^{\pi}E^{2i+1}F^{\pi}E^dF^{\pi}\\
&+&\sum\limits_{i=1}^{\infty}F^iF^{\pi}E^{2i-1}F^{\pi}E^dFF^{\pi}\big],\\
\end{array}}$$
$${\scriptsize\begin{array}{lll}
\Phi&=&\big[FF^{\pi}(E^d)^2+FF^{\pi}(E^d)^4FF^{\pi}+\sum\limits_{i=1}^{\infty}F^{i+1}F^{\pi}E^{2i+1}F^{\pi}E^d\\
&+&\sum\limits_{i=1}^{\infty}F^{i+1}F^{\pi}E^{2i-1}F^{\pi}E^dFF^{\pi}\big]\\
&&\big[F^{\pi}EE^d+F^{\pi}(E^d)^2FF^{\pi}+FF^{\pi}(E^d)^2+FF^{\pi}(E^d)^4FF^{\pi}\\
&+&\sum\limits_{i=1}^{\infty}F^{i+1}F^{\pi}(E^d)^{2i+2}+\sum\limits_{i=1}^{\infty}F^{i+1}F^{\pi}(E^d)^{2i+4}FF^{\pi}\big]\\
&+&\big[FF^{\pi}E^d+FF^{\pi}(E^d)^3FF^{\pi}-FF^{\pi}(E^d)^2F^{\pi}E-FF^{\pi}(E^d)^4FF^{\pi}E\\
&+&\sum\limits_{i=1}^{\infty}F^{i+1}F^{\pi}E^{2i+1}+\sum\limits_{i=1}^{\infty}F^{i+1}F^{\pi}E^{2i}E^dFF^{\pi}\\
&-&\sum\limits_{i=1}^{\infty}F^{i+1}F^{\pi}E^{2i+1}F^{\pi}E^dF^{\pi}E-\sum\limits_{i=1}^{\infty}F^{i+1}F^{\pi}E^{2i-1}F^{\pi}E^dFF^{\pi}E\big]\\
&&\big[F^{\pi}E^d+F^{\pi}(E^d)^3FF^{\pi}+\sum\limits_{i=1}^{\infty}F^iF^{\pi}E^{2i+1}+\sum\limits_{i=1}^{\infty}F^iF^{\pi}E^{2i}E^dFF^{\pi}\big],\\
\end{array}}$$
$${\scriptsize\begin{array}{lll}
\Omega&=&\big[FF^{\pi}(E^d)^2+FF^{\pi}(E^d)^4FF^{\pi}+\sum\limits_{i=1}^{\infty}F^{i+1}F^{\pi}E^{2i+1}F^{\pi}E^d\\
&+&\sum\limits_{i=1}^{\infty}F^{i+1}F^{\pi}E^{2i-1}F^{\pi}E^dFF^{\pi}\big]\\
&&\big[F^{\pi}E^dF^{\pi}+F^{\pi}(E^d)^3FF^{\pi}+FF^{\pi}(E^d)^3F^{\pi}+FF^{\pi}(E^d)^5FF^{\pi}\\
&+&\sum\limits_{i=1}^{\infty}F^{i+1}F^{\pi}(E^d)^{2i+3}F^{\pi}+\sum\limits_{i=1}^{\infty}F^{i+1}F^{\pi}(E^d)^{2i+5}FF^{\pi}\big]\\
&+&\big[FF^{\pi}E^d+F^{\pi}(E^d)^3FF^{\pi}-FF^{\pi}(E^d)^2F^{\pi}E-F^{\pi}(E^d)^4FF^{\pi}E\\
&+&\sum\limits_{i=1}^{\infty}F^{i+1}F^{\pi}E^{2i+1}+\sum\limits_{i=1}^{\infty}F^{i+1}F^{\pi}E^{2i}E^dFF^{\pi}\big]\\
&&\big[F^{\pi}(E^d)^2F^{\pi}+F^{\pi}(E^d)^4FF^{\pi}+\sum\limits_{i=1}^{\infty}F^iF^{\pi}E^{2i+1}F^{\pi}E^dF^{\pi}\\
&+&\sum\limits_{i=1}^{\infty}F^iF^{\pi}E^{2i-1}F^{\pi}E^dFF^{\pi}\big]~~~~~(**)\\
\end{array}}$$
\end{thm}
\begin{proof} Let $e=\left(
  \begin{array}{cc}
    F^{\pi}&0\\
    0&0
  \end{array}
\right)$. Then
$M=\left(
  \begin{array}{cc}
    a&b\\
    c&d
  \end{array}
\right)_e,$ where $$a=eMe, b=eM(I-e), c=(I-e)Me, d=(I-e)M(I-e).$$ Since $F^{\pi}EF^2=0$ and $F^{\pi}EFE=0$, we have
$$\begin{array}{c}
a=\left(
  \begin{array}{cc}
    F^{\pi}E&0\\
    0&0
  \end{array}
\right), b=\left(
  \begin{array}{cc}
    0&F^{\pi}\\
    0&0
  \end{array}
\right),\\
c=\left(
  \begin{array}{cc}
    FF^dEF^{\pi}&0\\
    FF^{\pi}&0
  \end{array}
\right),  d=\left(
  \begin{array}{cc}
    EFF^d&FF^d\\
    F^2F^d&0
  \end{array}
\right).
\end{array}$$ By direct computation, we get $$d^d=\left(
  \begin{array}{cc}
    0&F^d\\
    FF^d&-EF^d
  \end{array}
\right), d^{\pi}=\left(
  \begin{array}{cc}
    0&0\\
    0&F^{\pi}
  \end{array}
\right), dd^{\pi}=0.$$  As $EF^{\pi}$ has g-Drazin inverse, it follows by Cline's formula,
$F^{\pi}E$ has g-Drazin inverse. Since $F^{\pi}EFF^d=0$, it follows by~\cite[Lemma 2.2]{ZM} that $(F^{\pi}E)^d=F^{\pi}E^d$, and so $$a^d=\left(
  \begin{array}{cc}
    F^{\pi}E^d&0\\
    0&0
  \end{array}
\right), a^{\pi}=\left(
  \begin{array}{cc}
    F^{\pi}E^{\pi}&0\\
    0&0
  \end{array}
\right).$$ Write $M=P+Q$, where $$P=\left(
  \begin{array}{cc}
    a&b\\
    c&0
  \end{array}
\right)_e, Q=\left(
  \begin{array}{cc}
    0&0\\
    0&d
  \end{array}
\right)_e.$$

Step 1. We prove that$P$ has g-Drazin inverse. Clearly, $bd=0, bc=\left(
  \begin{array}{cc}
    FF^{\pi}&0\\
    0&0
  \end{array}
\right)$ has g-Drazin inverse, $(bc)^d=0$ and $(bc)^{\pi}=1$. We easily check that
$a(bc)^2=0, a(bc)a=0.$ In light of Theorem 2.3, $\left(
  \begin{array}{cc}
    a&1\\
    bc&0
  \end{array}
\right)$ has g-Drazin inverse. Moreover, we have
$$\left(
  \begin{array}{cc}
    a&1\\
    bc&0
  \end{array}
\right)^d=\left(\begin{array}{cc}
\Gamma&\Delta\\
 \Lambda&\Xi
\end{array}
\right),$$ where $${\scriptsize\begin{array}{l}
(\Gamma)_{11}=F^{\pi}E^d+F^{\pi}(E^d)^3FF^{\pi}+\sum\limits_{i=1}^{\infty}F^iF^{\pi}E^{2i+1}+\sum\limits_{i=1}^{\infty}F^iF^{\pi}E^{2i}E^dFF^{\pi},\\
(\Delta)_{11}=F^{\pi}(E^d)^2+F^{\pi}(E^d)^4FF^{\pi}+\sum\limits_{i=1}^{\infty}F^iF^{\pi}E^{2i+1}F^{\pi}E^d+\sum\limits_{i=1}^{\infty}F^iF^{\pi}E^{2i-1}F^{\pi}E^dFF^{\pi},\\
(\Lambda)_{11}=FF^{\pi}(E^d)^2+FF^{\pi}(E^d)^4FF^{\pi}+\sum\limits_{i=1}^{\infty}F^{i+1}F^{\pi}(E^d)^{2i+2}+\sum\limits_{i=1}^{\infty}F^{i+1}F^{\pi}(E^d)^{2i+4}FF^{\pi},\\
(\Xi)_{11}=FF^{\pi}(E^d)^3+FF^{\pi}(E^d)^5FF^{\pi}+\sum\limits_{i=1}^{\infty}F^{i+1}F^{\pi}(E^d)^{2i+3}+\sum\limits_{i=1}^{\infty}F^{i+1}F^{\pi}(E^d)^{2i+5}FF^{\pi},\\
(\Gamma)_{ij}=0, (\Delta)_{ij}=0, (\Lambda)_{ij}=0, (\Xi)_{ij}=0, otherwise.
\end{array}}$$
Since
$$\left(
  \begin{array}{cc}
  a&bc\\
1&0
\end{array}
\right)=\left(
  \begin{array}{cc}
    0 &1 \\
   1 & -a
  \end{array}
\right)^{-1}\left(
  \begin{array}{cc}
   a & 1 \\
   bc &0
  \end{array}
\right)\left(
  \begin{array}{cc}
    0 &1 \\
   1 & -a
  \end{array}
\right),$$ we see that $\left(
  \begin{array}{cc}
  a&bc\\
1&0
\end{array}
\right)$ has g-Drazin inverse and $$\left(
  \begin{array}{cc}
  a&bc\\
1&0
\end{array}
\right)^d=\left(
  \begin{array}{cc}
    0 &1 \\
   1 & -a
  \end{array}
\right)^{-1}\left(
  \begin{array}{cc}
   a & 1 \\
   bc &0
  \end{array}
\right)^d\left(
  \begin{array}{cc}
    0 &1 \\
   1 & -a
  \end{array}
\right);$$ hence, $$\left(
  \begin{array}{cc}
  a&bc\\
1&0
\end{array}
\right)^d=\left(
  \begin{array}{cc}
  a\Delta +\Xi&a\Gamma+\Lambda-a\Delta a-\Xi a\\
\Delta&\Gamma-\Delta a
\end{array}
\right).$$
Clearly, we have $$\left(
  \begin{array}{cc}
    a&bc\\
    1&0
  \end{array}
\right)=\left(
  \begin{array}{cc}
    a&b\\
    1&0
  \end{array}
\right)\left(
  \begin{array}{cc}
    1&0\\
    0&c
  \end{array}
\right).$$ By using Cline's formula,
$$\begin{array}{lll}
P^d&=&\left(
  \begin{array}{cc}
   1&0\\
    0&c
  \end{array}
\right)\big(\left(
  \begin{array}{cc}
  a&bc\\
1&0
\end{array}
\right)^d\big)^2\left(
  \begin{array}{cc}
     a&b\\
    1&0
  \end{array}
\right)\\
&=&\left(
  \begin{array}{cc}
   a\Delta +\Xi&a\Gamma+\Lambda-a\Delta a-\Xi a\\
c\Delta&c\Gamma-c\Delta a
  \end{array}
\right)\left(
  \begin{array}{cc}
   a\Gamma+\Lambda&a\Delta b+\Xi b\\
\Gamma&\Delta b
  \end{array}
\right)\\
&=&\left(\begin{array}{cc}
\gamma&\delta \\
 \lambda&\xi
\end{array}
\right),
\end{array}$$ where $$\begin{array}{lll}
\gamma&=&(a\Delta +\Xi)(a\Gamma+\Lambda)+(a\Gamma+\Lambda-a\Delta a-\Xi a)\Gamma,\\
\delta&=&(a\Delta +\Xi)^2b+(a\Gamma+\Lambda-a\Delta a-\Xi a)\Delta b,\\
\lambda&=&c\Delta (a\Gamma+\Lambda)+c(\Gamma-\Delta a)\Gamma,\\
\xi&=&c\Delta (a\Delta +\Xi )b+c(\Gamma-\Delta a)\Delta b.
\end{array}$$ We directly compute that $$\begin{array}{c}
a\Delta=\left(
  \begin{array}{cc}
  F^{\pi}E^d+F^{\pi}(E^d)^3FF^{\pi}&0\\
0&0
\end{array}
\right),\\
a\Gamma=\left(
  \begin{array}{cc}
  F^{\pi}EE^d+F^{\pi}(E^d)^2FF^{\pi}&0\\
0&0
\end{array}
\right).
\end{array}$$
Hence we have $$\begin{array}{lll}
(a\Delta+\Xi)_{11}&=&F^{\pi}E^d+F^{\pi}(E^d)^3FF^{\pi}+FF^{\pi}(E^d)^3+FF^{\pi}(E^d)^5FF^{\pi}\\
&+&\sum\limits_{i=1}^{\infty}F^{i+1}F^{\pi}(E^d)^{2i+3}+\sum\limits_{i=1}^{\infty}F^{i+1}F^{\pi}(E^d)^{2i+5}FF^{\pi},\\
(a\Delta+\Xi)_{ij}&=&0, ~~\mbox{otherwise};\\
(a\Gamma+\Lambda)_{11}&=&F^{\pi}EE^d+F^{\pi}(E^d)^2FF^{\pi}+FF^{\pi}(E^d)^2+FF^{\pi}(E^d)^4FF^{\pi}\\
&+&\sum\limits_{i=1}^{\infty}F^{i+1}F^{\pi}(E^d)^{2i+2}+\sum\limits_{i=1}^{\infty}F^{i+1}F^{\pi}(E^d)^{2i+4}FF^{\pi},\\
(a\Gamma+\Lambda)_{ij}&=&0, ~~\mbox{otherwise.}
\end{array}$$
Also we have $$\begin{array}{ll}
&(a\Gamma+\Lambda-a\Delta a-\Xi a)_{11}\\
=&F^{\pi}EE^d+F^{\pi}(E^d)^2FF^{\pi}+FF^{\pi}(E^d)^2+FF^{\pi}(E^d)^4FF^{\pi}\\
-&F^{\pi}E^dF^{\pi}E-F^{\pi}(E^d)^3FF^{\pi}E-FF^{\pi}(E^d)^3F^{\pi}E-FF^{\pi}(E^d)^5FF^{\pi}E\\
+&\sum\limits_{i=1}^{\infty}F^{i+1}F^{\pi}(E^d)^{2i+2}+\sum\limits_{i=1}^{\infty}F^{i+1}F^{\pi}(E^d)^{2i+4}FF^{\pi}\\
-&\sum\limits_{i=1}^{\infty}F^{i+1}F^{\pi}(E^d)^{2i+3}F^{\pi}E-\sum\limits_{i=1}^{\infty}F^{i+1}F^{\pi}(E^d)^{2i+5}FF^{\pi}E,\\
&(a\Gamma+\Lambda-a\Delta a-\Xi a)_{ij}=0,~~\mbox{otherwise}.
\end{array}$$
Therefore $${\scriptsize\begin{array}{lll}
(\gamma)_{11}&=&\big[F^{\pi}E^d+F^{\pi}(E^d)^3FF^{\pi}+FF^{\pi}(E^d)^3+FF^{\pi}(E^d)^5FF^{\pi}\\
&+&\sum\limits_{i=1}^{\infty}F^{i+1}F^{\pi}(E^d)^{2i+3}+\sum\limits_{i=1}^{\infty}F^{i+1}F^{\pi}(E^d)^{2i+5}FF^{\pi}\big]\\
&&\big[F^{\pi}EE^d+F^{\pi}(E^d)^2FF^{\pi}+FF^{\pi}(E^d)^2+FF^{\pi}(E^d)^4FF^{\pi}\\
&+&\sum\limits_{i=1}^{\infty}F^{i+1}F^{\pi}(E^d)^{2i+2}+\sum\limits_{i=1}^{\infty}F^{i+1}F^{\pi}(E^d)^{2i+4}FF^{\pi}\big]\\
&+&\big[F^{\pi}EE^d+F^{\pi}(E^d)^2FF^{\pi}+FF^{\pi}(E^d)^2+FF^{\pi}(E^d)^4FF^{\pi}\\
&-&F^{\pi}E^dF^{\pi}E-F^{\pi}(E^d)^3FF^{\pi}E-FF^{\pi}(E^d)^3F^{\pi}E-FF^{\pi}(E^d)^5FF^{\pi}E\\
&+&\sum\limits_{i=1}^{\infty}F^{i+1}F^{\pi}(E^d)^{2i+2}+\sum\limits_{i=1}^{\infty}F^{i+1}F^{\pi}(E^d)^{2i+4}FF^{\pi}\\
&-&\sum\limits_{i=1}^{\infty}F^{i+1}F^{\pi}(E^d)^{2i+3}F^{\pi}E-\sum\limits_{i=1}^{\infty}F^{i+1}F^{\pi}(E^d)^{2i+5}FF^{\pi}E\big]\\
&&\big[F^{\pi}E^d+F^{\pi}(E^d)^3FF^{\pi}+\sum\limits_{i=1}^{\infty}F^iF^{\pi}E^{2i+1}+\sum\limits_{i=1}^{\infty}F^iF^{\pi}E^{2i}E^dFF^{\pi}\big],\\
&&(\gamma)_{ij}=0,~~\mbox{otherwise}.
\end{array}}$$
We see that $${\scriptsize\begin{array}{rll}
(c\Delta)_{11}&=&FF^dEF^{\pi}(E^d)^2+FF^dEF^{\pi}(E^d)^4FF^{\pi}+\sum\limits_{i=1}^{\infty}FF^dEF^iF^{\pi}E^{2i+1}F^{\pi}E^d\\
&+&\sum\limits_{i=1}^{\infty}FF^dEF^iF^{\pi}E^{2i-1}F^{\pi}E^dFF^{\pi},\\
(c\Delta)_{21}&=&FF^{\pi}(E^d)^2+FF^{\pi}(E^d)^4FF^{\pi}+\sum\limits_{i=1}^{\infty}F^{i+1}F^{\pi}E^{2i+1}F^{\pi}E^d\\
&+&\sum\limits_{i=1}^{\infty}F^{i+1}F^{\pi}E^{2i-1}F^{\pi}E^dFF^{\pi},\\
(c\Delta)_{ij}&=&0,~~\mbox{otherwise};\\
(\Delta b)_{12}&=&F^{\pi}(E^d)^2F^{\pi}+F^{\pi}(E^d)^4FF^{\pi}+\sum\limits_{i=1}^{\infty}F^iF^{\pi}E^{2i+1}F^{\pi}E^dF^{\pi}\\
&+&\sum\limits_{i=1}^{\infty}F^iF^{\pi}E^{2i-1}F^{\pi}E^dFF^{\pi},\\
(\Delta b)_{ij}&=&0,~~\mbox{otherwise};\\
\end{array}}$$

$${\scriptsize\begin{array}{lll}
(\Gamma-\Delta a)_{11}&=&F^{\pi}E^d+F^{\pi}(E^d)^3FF^{\pi}-F^{\pi}(E^d)^2F^{\pi}E-F^{\pi}(E^d)^4FF^{\pi}E\\
&+&\sum\limits_{i=1}^{\infty}F^iF^{\pi}E^{2i+1}+\sum\limits_{i=1}^{\infty}F^iF^{\pi}E^{2i}E^dFF^{\pi}\\
&-&\sum\limits_{i=1}^{\infty}F^iF^{\pi}E^{2i+1}F^{\pi}E^dF^{\pi}E-\sum\limits_{i=1}^{\infty}F^iF^{\pi}E^{2i-1}F^{\pi}E^dFF^{\pi}E,\\
(\Gamma-\Delta a)_{ij}&=&0,~~\mbox{otherwise};\\
\end{array}}$$
Thus, $${\scriptsize\begin{array}{lll}
(\delta)_{12}&=&\big[F^{\pi}E^dF^{\pi}+F^{\pi}(E^d)^3FF^{\pi}+FF^{\pi}(E^d)^3F^{\pi}+FF^{\pi}(E^d)^5FF^{\pi}\\
&+&\sum\limits_{i=1}^{\infty}F^{i+1}F^{\pi}(E^d)^{2i+3}F^{\pi}+\sum\limits_{i=1}^{\infty}F^{i+1}F^{\pi}(E^d)^{2i+5}FF^{\pi}\big]^2\\
&+&\big[F^{\pi}EE^d+F^{\pi}(E^d)^2FF^{\pi}+FF^{\pi}(E^d)^2+FF^{\pi}(E^d)^4FF^{\pi}\\
&-&F^{\pi}E^dF^{\pi}E-F^{\pi}(E^d)^3FF^{\pi}E-FF^{\pi}(E^d)^3F^{\pi}E-FF^{\pi}(E^d)^5FF^{\pi}E\\
&+&\sum\limits_{i=1}^{\infty}F^{i+1}F^{\pi}(E^d)^{2i+2}+\sum\limits_{i=1}^{\infty}F^{i+1}F^{\pi}(E^d)^{2i+4}FF^{\pi}\\
&-&\sum\limits_{i=1}^{\infty}F^{i+1}F^{\pi}(E^d)^{2i+3}F^{\pi}E-\sum\limits_{i=1}^{\infty}F^{i+1}F^{\pi}(E^d)^{2i+5}FF^{\pi}E\big]\\
&&\big[F^{\pi}(E^d)^2F^{\pi}+F^{\pi}(E^d)^4FF^{\pi}+\sum\limits_{i=1}^{\infty}F^iF^{\pi}E^{2i+1}F^{\pi}E^dF^{\pi}\\
&+&\sum\limits_{i=1}^{\infty}F^iF^{\pi}E^{2i-1}F^{\pi}E^dFF^{\pi}\big],\\
&&(\delta)_{ij}=0,~~\mbox{otherwise}.
\end{array}}$$
Moreover, we have $${\scriptsize\begin{array}{lll}
(\lambda)_{11}&=&\big[FF^dEF^{\pi}(E^d)^2+FF^dEF^{\pi}(E^d)^4FF^{\pi}+\sum\limits_{i=1}^{\infty}FF^dEF^iF^{\pi}E^{2i+1}F^{\pi}E^d\\
&+&\sum\limits_{i=1}^{\infty}FF^dEF^iF^{\pi}E^{2i-1}F^{\pi}E^dFF^{\pi}\big]\\
&&\big[F^{\pi}EE^d+F^{\pi}(E^d)^2FF^{\pi}+FF^{\pi}(E^d)^2+FF^{\pi}(E^d)^4FF^{\pi}\\
&+&\sum\limits_{i=1}^{\infty}F^{i+1}F^{\pi}(E^d)^{2i+2}+\sum\limits_{i=1}^{\infty}F^{i+1}F^{\pi}(E^d)^{2i+4}FF^{\pi}\big]\\
&+&\big[FF^dEF^{\pi}E^d+FF^dEF^{\pi}(E^d)^3FF^{\pi}-FF^dEF^{\pi}(E^d)^2F^{\pi}E\\
&-&FF^dEF^{\pi}(E^d)^4FF^{\pi}E+\sum\limits_{i=1}^{\infty}FF^dEF^iF^{\pi}E^{2i+1}+\sum\limits_{i=1}^{\infty}FF^dEF^iF^{\pi}E^{2i}E^dFF^{\pi}\\
&-&\sum\limits_{i=1}^{\infty}FF^dEF^iF^{\pi}E^{2i+1}F^{\pi}E^dF^{\pi}E-\sum\limits_{i=1}^{\infty}FF^dEF^iF^{\pi}E^{2i-1}F^{\pi}E^dFF^{\pi}E\big]\\
&&\big[F^{\pi}E^d+F^{\pi}(E^d)^3FF^{\pi}+\sum\limits_{i=1}^{\infty}F^iF^{\pi}E^{2i+1}+\sum\limits_{i=1}^{\infty}F^iF^{\pi}E^{2i}E^dFF^{\pi}\big]
\end{array}}$$
$${\scriptsize\begin{array}{lll}
(\lambda)_{21}&=&\big[FF^{\pi}(E^d)^2+FF^{\pi}(E^d)^4FF^{\pi}+\sum\limits_{i=1}^{\infty}F^{i+1}F^{\pi}E^{2i+1}F^{\pi}E^d\\
&+&\sum\limits_{i=1}^{\infty}F^{i+1}F^{\pi}E^{2i-1}F^{\pi}E^dFF^{\pi}\big]\\
&&\big[F^{\pi}EE^d+F^{\pi}(E^d)^2FF^{\pi}+FF^{\pi}(E^d)^2+FF^{\pi}(E^d)^4FF^{\pi}\\
&+&\sum\limits_{i=1}^{\infty}F^{i+1}F^{\pi}(E^d)^{2i+2}+\sum\limits_{i=1}^{\infty}F^{i+1}F^{\pi}(E^d)^{2i+4}FF^{\pi}\big]\\
&+&\big[FF^{\pi}E^d+FF^{\pi}(E^d)^3FF^{\pi}-FF^{\pi}(E^d)^2F^{\pi}E-FF^{\pi}(E^d)^4FF^{\pi}E\\
&+&\sum\limits_{i=1}^{\infty}F^{i+1}F^{\pi}E^{2i+1}+\sum\limits_{i=1}^{\infty}F^{i+1}F^{\pi}E^{2i}E^dFF^{\pi}\\
&-&\sum\limits_{i=1}^{\infty}F^{i+1}F^{\pi}E^{2i+1}F^{\pi}E^dF^{\pi}E-\sum\limits_{i=1}^{\infty}F^{i+1}F^{\pi}E^{2i-1}F^{\pi}E^dFF^{\pi}E\big]\\
&&\big[F^{\pi}E^d+F^{\pi}(E^d)^3FF^{\pi}+\sum\limits_{i=1}^{\infty}F^iF^{\pi}E^{2i+1}+\sum\limits_{i=1}^{\infty}F^iF^{\pi}E^{2i}E^dFF^{\pi}\big],\\
&&(\lambda)_{ij}=0,~~\mbox{otherwise}.
\end{array}}$$
Furthermore, we get $${\scriptsize\begin{array}{lll}
(\xi)_{12}&=&\big[FF^dEF^{\pi}(E^d)^2+FF^dEF^{\pi}(E^d)^4FF^{\pi}+\sum\limits_{i=1}^{\infty}FF^dEF^iF^{\pi}E^{2i+1}F^{\pi}E^d\\
&+&\sum\limits_{i=1}^{\infty}FF^dEF^iF^{\pi}E^{2i-1}F^{\pi}E^dFF^{\pi}\big]\\
&&\big[F^{\pi}E^dF^{\pi}+F^{\pi}(E^d)^3FF^{\pi}+FF^{\pi}(E^d)^3F^{\pi}+FF^{\pi}(E^d)^5FF^{\pi}\\
&+&\sum\limits_{i=1}^{\infty}F^{i+1}F^{\pi}(E^d)^{2i+3}F^{\pi}+\sum\limits_{i=1}^{\infty}F^{i+1}F^{\pi}(E^d)^{2i+5}FF^{\pi}\big]\\
&+&\big[FF^dEF^{\pi}E^d+FF^dEF^{\pi}(E^d)^3FF^{\pi}-FF^dEF^{\pi}(E^d)^2F^{\pi}E\\
&-&FF^dEF^{\pi}(E^d)^4FF^{\pi}E+\sum\limits_{i=1}^{\infty}FF^dEF^iF^{\pi}E^{2i+1}+\sum\limits_{i=1}^{\infty}FF^dEFF^iF^{\pi}E^{2i}E^dFF^{\pi}\\
&-&\sum\limits_{i=1}^{\infty}FF^dEF^iF^{\pi}E^{2i+1}F^{\pi}E^dF^{\pi}E-\sum\limits_{i=1}^{\infty}FF^dEF^iF^{\pi}E^{2i-1}F^{\pi}E^dFF^{\pi}E\big]\\
&&\big[F^{\pi}(E^d)^2F^{\pi}+F^{\pi}(E^d)^4FF^{\pi}+\sum\limits_{i=1}^{\infty}F^iF^{\pi}E^{2i+1}F^{\pi}E^dF^{\pi}\\
&+&\sum\limits_{i=1}^{\infty}F^iF^{\pi}E^{2i-1}F^{\pi}E^dFF^{\pi}\big]\end{array}}$$
$${\scriptsize\begin{array}{lll}
(\xi)_{22}&=&\big[FF^{\pi}(E^d)^2+FF^{\pi}(E^d)^4FF^{\pi}+\sum\limits_{i=1}^{\infty}F^{i+1}F^{\pi}E^{2i+1}F^{\pi}E^d\\
&+&\sum\limits_{i=1}^{\infty}F^{i+1}F^{\pi}E^{2i-1}F^{\pi}E^dFF^{\pi}\big]\\
&&\big[F^{\pi}E^dF^{\pi}+F^{\pi}(E^d)^3FF^{\pi}+FF^{\pi}(E^d)^3F^{\pi}+FF^{\pi}(E^d)^5FF^{\pi}\\
&+&\sum\limits_{i=1}^{\infty}F^{i+1}F^{\pi}(E^d)^{2i+3}F^{\pi}+\sum\limits_{i=1}^{\infty}F^{i+1}F^{\pi}(E^d)^{2i+5}FF^{\pi}\big]\\
&+&\big[FF^{\pi}E^d+F^{\pi}(E^d)^3FF^{\pi}-FF^{\pi}(E^d)^2F^{\pi}E-F^{\pi}(E^d)^4FF^{\pi}E\\
&+&\sum\limits_{i=1}^{\infty}F^{i+1}F^{\pi}E^{2i+1}+\sum\limits_{i=1}^{\infty}F^{i+1}F^{\pi}E^{2i}E^dFF^{\pi}\big]\\
&&\big[F^{\pi}(E^d)^2F^{\pi}+F^{\pi}(E^d)^4FF^{\pi}+\sum\limits_{i=1}^{\infty}F^iF^{\pi}E^{2i+1}F^{\pi}E^dF^{\pi}\\
&+&\sum\limits_{i=1}^{\infty}F^iF^{\pi}E^{2i-1}F^{\pi}E^dFF^{\pi}\big],\\
&&(\xi)_{ij}=0,~~\mbox{otherwise}.
\end{array}}$$

Step 2. We claim that $Q$ has g-Drazin inverse. We have $$Q^d=\left(
  \begin{array}{cc}
    0&0\\
    0&d^d
  \end{array}
\right)_e~\mbox{and}~Q^{\pi}=\left(
  \begin{array}{cc}
    e&0\\
    0&d^{\pi}
  \end{array}
\right)_e.$$

Step 3. Since $bd=0,$ we have $PQ=0$. Clearly, we have

$$\begin{array}{lll}
Q^{\pi}P^d&=&\left(\begin{array}{cc}
\gamma&\delta \\
 d^{\pi}\lambda&d^{\pi}\xi
\end{array}
\right)_e\\
&=&\gamma+\delta+d^{\pi}\lambda+d^{\pi}\xi\\
&=&\left(\begin{array}{cc}
\gamma_{11}&\delta_{12} \\
 \lambda_{21}&\xi_{22}
\end{array}
\right)
\end{array}$$

By direct computation, we have
$$\begin{array}{lll}
(a+c)(\gamma+\delta)&=&\left(\begin{array}{cc}
E-FF^dEFF^d&0 \\
FF^{\pi}&0
\end{array}
\right)\left(\begin{array}{cc}
\gamma_{11}&\delta_{12}\\
0&0
\end{array}
\right)\\
&=&\left(\begin{array}{cc}
E\gamma_{11}&E\delta_{12}\\
F\gamma_{11}&F\delta_{12}\\
\end{array}
\right)
\end{array}$$
and
$$\begin{array}{lll}
b(\gamma+\delta)&=&\left(\begin{array}{cc}
0&F^{\pi} \\
0&0
\end{array}
\right)\left(\begin{array}{cc}
\gamma_{11}&\delta_{12}\\
0&0
\end{array}
\right)\\
&=&0.
\end{array}$$
Hence,
$$\begin{array}{lll}
PP^d&=&\left(\begin{array}{cc}
a&b\\
c&0
\end{array}
\right)_e\left(\begin{array}{cc}
\gamma&\delta \\
 \lambda&\xi
\end{array}
\right)_e\\
&=&(a+c)(\gamma+\delta)+b(\lambda+\xi)\\
&=&\left(\begin{array}{cc}
E\gamma_{11}&E\delta_{12}\\
F\gamma_{11}&F\delta_{12}\\
\end{array}
\right).
\end{array}$$
Also we see that $Q^{\pi}Q^i=0$ for any $i\geq 1$. In view of~\cite[Theorem 2.3]{D}, we have
$$\begin{array}{lll}
M^d&=&Q^{\pi}\sum\limits_{i=0}^{\infty}Q^i(P^d)^{i+1}+\sum\limits_{i=0}^{\infty}(Q^d)^{i+1}P^iP^{\pi}\\
&=&\left(
  \begin{array}{cc}
   \Theta&\Psi\\
    \Phi&\Omega
  \end{array}
\right)+\sum\limits_{i=0}^{\infty}\left(
  \begin{array}{cc}
    0&F^d\\
   FF^d&-EF^d
  \end{array}
\right)^{i+1}\\
&&\left(
  \begin{array}{cc}
   F^{\pi}E+FF^dEF^{\pi}&F^{\pi}\\
    FF^{\pi}&0
  \end{array}
\right)^i\left(
  \begin{array}{cc}
   I-E\Theta&-E\Psi\\
    -F\Theta&I-F\Psi
  \end{array}
\right),
\end{array}$$ where $\Theta=\gamma_{11}, \Psi=\delta_{12}, \Phi=\lambda_{21}, \Omega=\xi_{22},$ as asserted.\end{proof}

\begin{cor} Let $E$ and $F$ have g-Drazin inverses. If $F^{\pi}EF^2=0$ and $F^{\pi}EFE=0$, then $M=\left(
  \begin{array}{cc}
    E&F\\
    I&0
  \end{array}
\right)$ has g-Drazin inverse. In this case, $$\begin{array}{lll}
M^d&=&\left(
  \begin{array}{cc}
   E\Psi+\Omega&E\Theta-E\Psi E+\Phi-\Omega E\\
    \Psi&\Theta-\Psi E
  \end{array}
\right)+\sum\limits_{i=0}^{\infty}\left(
  \begin{array}{cc}
    FF^d&0\\
   0&F^d
  \end{array}
\right)\\
&&\left(
  \begin{array}{cc}
    0&F^d\\
   FF^d&-EF^d
  \end{array}
\right)^{i}\left(
  \begin{array}{cc}
   F^{\pi}E+FF^dEF^{\pi}&F^{\pi}\\
    FF^{\pi}&0
  \end{array}
\right)^i\\
&&\left(
  \begin{array}{cc}
   -E\Psi&I-E\Theta +E\Psi E\\
    I-F\Psi&-E-F\Theta+F\Psi E
  \end{array}
\right),
\end{array}$$ where $\Theta, \Psi, \Phi, \Omega$ as given in $(**)$.\end{cor}
\begin{proof} Obviously, we have
$$\left(
  \begin{array}{cc}
  E&F\\
I&0
\end{array}
\right)=\left(
  \begin{array}{cc}
    0 &I \\
   I & -E
  \end{array}
\right)^{-1}\left(
  \begin{array}{cc}
   E & I \\
   F &0
  \end{array}
\right)\left(
  \begin{array}{cc}
    0 &I \\
   I & -E
  \end{array}
\right),$$ and so
$$\left(
  \begin{array}{cc}
  E&F\\
I&0
\end{array}
\right)^d=\left(
  \begin{array}{cc}
    E &I \\
   I & 0
  \end{array}
\right)\left(
  \begin{array}{cc}
   E & I \\
   F &0
  \end{array}
\right)^d\left(
  \begin{array}{cc}
    0 &I \\
   I & -E
  \end{array}
\right).$$
Applying Theorem 2.6 to the matrix $\left(
  \begin{array}{cc}
   E &I \\
   F &0
  \end{array}
\right)$, we obtain the result.\end{proof}

\section*{Acknowledgement}

The authors would like to thank the referee for his/her helpful suggestions for the improvement of this paper.

\end{document}